\documentclass{amsart}
\usepackage[cp1250]{inputenc}
\usepackage[T1]{fontenc}
\usepackage{amsmath}
\usepackage{amssymb}
\usepackage{amsthm}
\usepackage{amsfonts}
\usepackage{mathrsfs}
\newcommand{\zz}{\mathbf{\mathbb{Z}}}
\newcommand{\rr}{\mathbf{\mathbb{R}}}
\newcommand{\er}{\mathbf{\mathbb{R}}}
\newcommand{\cc}{\mathbf{\mathbb{C}}}

\newcommand{\V}{\mathbf{Y}}
\newcommand{\W}{\mathbf{\mathbb{L}}}
\newcommand{\X}{\mathbf{\mathbb{X}}}
\newcommand{\Y}{\mathbf{\mathbb{Y}}}
\newcommand{\M}{\mathbf{\mathbb{M}}}

\newcommand{\E}{\mathbf{L}}
\newcommand{\C}{{\mathscr{C}}}
\newcommand{\Z}{{\mathbf{\mathcal{Z}}}}
\newcommand{\FF}{{\mathbf{\mathcal{F}}}}

\newcommand{\EE}{{\mathbf{X}}}

\newcommand{\sign}{\,{\rm sign}\,}

\newcommand{\ord}{\operatorname{ord}}
\newcommand{\Int}{\operatorname{Int}}
\newcommand{\ve}{\varepsilon}

\newcommand{\vf}{\varphi}

\newcommand{\grad}{\nabla}

\newcommand{\dist}{\operatorname{dist}}
\newcommand{\length}{\operatorname{length}}

\numberwithin{equation}{section}
\newcounter{thm}[section]

\newtheorem{lem}[thm]{Lemma}
\newtheorem{twr}[thm]{Theorem}

\newtheorem{cor}[thm]{Corollary}
\newtheorem{prop}[thm]{Proposition}
\newtheorem{fact}[thm]{Fact}

\theoremstyle{definition}
\newtheorem{remark}[thm]{Remark}

\numberwithin{equation}{section}
\def \wll{{\mathcal{L}}}

\def\FF{\mathscr{F}}

\def\la{\lambda}

\newcommand{\wykl}{{\mathcal{L}}}

\begin{document}

\title[Effective {\L}ojasiewicz gradient inequality]{Effective {\L}ojasiewicz gradient inequality\\ and finite determinacy of non-isolated Nash function singularities}


\author[B. Osi\'nska-Ulrych]{Beata Osi\'nska-Ulrych}
\address{Beata Osi\'nska-Ulrych, Faculty of Mathematics and Computer Science, University of \L \'od\'z, 
S. Banacha 22, 90-238 \L \'od\'z, Poland}
\email{beata.osinska@wmii.uni.lodz.pl}

\author[G. Skalski]{Grzegorz Skalski}
\address{Grzegorz Skalski, Faculty of Mathematics and Computer Science, University of \L \'od\'z, 
S. Banacha 22, 90-238 \L \'od\'z, Poland}
\email{grzegorz.skalski@wmii.uni.lodz.pl}

\author[S. Spodzieja]{Stanis{\l}aw Spodzieja}
\address{Stanis{\l}aw Spodzieja, Faculty of Mathematics and Computer Science, University of \L \'od\'z, 
S. Banacha 22, 90-238 \L \'od\'z, Poland}
\email{spodziej@math.uni.lodz.pl}

\date{\today}

\begin{abstract} 
Let $X\subset \rr^n$ be a compact semialgebraic set and let $f:X\to \rr$ be a nonzero Nash function. We give a  Solern\'o and  D'Acunto-Kurdyka type estimation of the exponent $\varrho\in[0,1)$ in the 
{\L}ojasiewicz gradient inequality $|\grad f(x)|\ge C|f(x)|^\varrho$ for $x\in X$, $|f(x)|<\varepsilon$ for some constants $C,\varepsilon>0$, in terms of the degree of a polynomial $P$ such that $P(x,f(x))=0$, $x\in X$. As a corollary we obtain an estimation of the degree of sufficiency of non-isolated Nash functions singularities.
%
\end{abstract}

\maketitle

\section{Introduction}

{\L}ojasiewicz inequalities are important tools in various branches of mathematics: differential equations, singularity theory and optimization (for more detailed references, see for example  \cite{KMP},  \cite{KS1}, \cite{KS2}, \cite{LejeuneTeissier} and 
\cite{RS3}). 
 Quantitative aspects, like estimates  (or exact computation), of  these exponents are subject of intensive study in real and complex algebraic geometry (see for instance \cite{KS1}, \cite{KS2}, \cite{KSS} and \cite{RS2}).  
Our main goal  is to give, in terms of the {\L}ojasiewicz inequality, an effective sufficient condition for Nash function germs of non-isolated singularity at zero to be isotopical (Theorem \ref{mainsuffjetLojineq}). The main tool in the proof is an effective estimation of the exponent in the {\L}ojasiewicz gradient inequality (Theorems \ref{maintwr} and \ref{maintwrIII}). 

Determinacy of jets of functions with isolated singularity at zero was investigated by many authors, including  N. H.~Kuiper \cite{Kui},  T. C.~Kuo \cite{Kuo}, J.~Bochnak and S. \L ojasiewicz \cite{BL} for real functions and  S. H.~Chang and Y. C.~Lu \cite{CL}, B.~Teissier \cite{T} and J. Bochnak and W.~Kucharz \cite{BK} for complex functions. Similar investigations were also carried out for functions in a neighbourhood of infinity by P. Cassou-Nogu\`es and  H. H. Vui \cite{CH} (see also \cite{RS4}, \cite{Sk}). The case of real jets with non-isolated singularities was studied among others by  V. Grandjean \cite{G} and 
X. Xu \cite{Xu}, and for complex functions by D.~Siersma \cite{Si2} and R. Pellikaan \cite{P}. In the case of nondegenerate analytic functions $f$, $g$, a condition for topological triviality of deformations $f+tg$, $t\in[0,1]$ in terms of Newton polyhedra was obtained by J. Damon and T. Gaffney \cite{DG}, and for blow analytic triviality  by  T. Fukui and E.~~Yoshinaga \cite{FY}. Some algebraic conditions for finite determinacy of a smooth function jet were obtained by L. Kushner \cite{Kne}.


\subsection{{\L}ojasiewicz gradient inequality}

Let $U\subset \rr^n$ be an open set and let $a\in U$. 
 Let $f,F: U \to \rr$ be continuous semialgebraic functions such that $a\in F^{-1}(0) \subset f^{-1}(0)\subset U$. 
Then 
 the following \emph{{\L}ojasiewicz inequality} holds:
\begin{equation}\label{Lojineq01}
|F(x)|\ge C|f(x)|^\eta \;  \hbox{ in a neighbourhood of $a\in\rr^n$ for some constant $C>0$.}
\end{equation}
The lower bound of the exponents $\eta$ in \eqref{Lojineq01} is called the \emph{{\L}ojasiewicz exponent} of the pair  $(F,f)$  \emph{at} $a$ and is denoted by $\wykl_a(F,f)$. It is known   that $\wykl_a(F,f)$ is  a rational number (see \cite{BR}) and the
inequality \eqref{Lojineq01} holds actually with $\eta= \wykl_a(F,f)$ on some neighbourhood of the point $a$ for some positive constant $C$
 (see for instance \cite{Sp1}).
An asymptotic estimate  for $\wykl_a(F,f)$  was obtained by Solern\'o \cite{Solerno}:
\begin{equation}\label{eqsolerno}\tag{S}
\wykl_a(F,f) \le D^{M^{c\ell}},
\end{equation}
where $D$ is a bound for the degrees of the polynomials involved in a description of $F$, $f$ and $U$; $M$ is the number of variables
in these formulas; $\ell$ is the maximum  number of alternating blocs of quantifiers in these formulas; and $c$ is an unspecified universal constant.

In this paper, we consider the case when $F$ is equal to the gradient $\nabla f:=\left(\frac{\partial f}{\partial x_1},\ldots,\frac{\partial f}{\partial x_n}\right):U\to\er^n $ of a Nash function $f$ in $x=(x_1,\ldots,x_n)$. Recall that semialgebraic and analytic functions are called \emph{Nash functions}. 

Our main  goal is to obtain an effective estimate for the  exponent $\varrho\in[0,1)$ in the following \emph{{\L}ojasiewicz gradient inequality} (see \cite{Lo1} or \cite{Lo2}, cf. \cite{T}):
\begin{equation}\label{Lojineq1}\tag{{\rm \L}}
|\nabla f(x)|\ge C|f(x)|^\varrho  \hbox{ in a neighbourhood of $a\in\rr^n$ for some constant $C>0$}
\end{equation}
for an arbitrary Nash function $f:U\to \er$, where $f(a)=0$, in terms of the degree of a polynomial $P\in\rr[x,y]$ describing the graph of $f$. 
 We denote by $|\nabla f(x)|$ the Euclidean norm of $\nabla f(x)$, i.e. $|\nabla f(x)|^2=\left(\frac{\partial f}{\partial x_1}(x)\right)^2+\cdots+\left(\frac{\partial f}{\partial x_n}(x)\right)^2$.

The smallest  exponent $\varrho$ in  \eqref{Lojineq1}, denoted by $\varrho_a(f)$,  is called the \emph{{\L}ojasiewicz exponent in the gradient inequality} at $a$. It is known that \eqref{Lojineq1} holds with $\varrho=\varrho_a(f)$.

In the case of a polynomial function $f:\er^n\to\er$ of degree $d>0$ such that $0$ is an isolated point of $f^{-1}(0)$, J.~Gwo\'zdziewicz \cite{Gw} (cf. \cite{K1}) proved that
\begin{equation}\label{Gw2}\tag{G2}
\varrho_0(f)\le 1-\frac{1}{(d-1)^n+1},
\end{equation}
and in the general case of an arbitrary polynomial $f$, 
 D. D'Acunto and K.~Kurdyka \cite{DK} (cf. \cite{DK1}, \cite{Gabrielov} and \cite{TienSon}) showed that
\begin{equation}\label{DKineq}\tag{DK}
\varrho_0(f)\le 1-\frac{1}{d(3d-3)^{n-1}},\quad\hbox{provided}\quad d\ge 2.
\end{equation}
If $f$ is a rational function of the form $f={p}\slash{q}$, 
 where $p,\,q\in \rr[x]$, $p(0)=0$ and $q(0)\ne 0$, then  $\varrho_0(f)=\varrho_0(p)$, so \eqref{Gw2} and \eqref{DKineq} hold with $d=\deg p$.

The aim of this paper is to show generalizations of the above estimates 
 for Nash functions 
 (see Theorems \ref{maintwr} and \ref{maintwrIII} in Section \ref{Lojgradineqsect}). More precisely, let $U\subset \rr^n$ be a neighbourhood of $a\in\rr^n$ and let $f:U\to \rr$ be a nonzero Nash function. We give a  Solern\'o and  D'Acunto-Kurdyka type estimation of the exponent $\varrho\in[0,1)$ in the 
{\L}ojasiewicz gradient inequality \eqref{Lojineq1} in terms of the degree $d$ of a nonzero polynomial $P$ such that $P(x,f(x))=0$, $x\in U$. Namely, in Theorem \ref{maintwrIII} we obtain
$$
\varrho_a(f)\leq 1-\frac{1}{2(2d-1)^{3n+1}}.
$$
If additionally $n\ge 2$ and  $\frac{\partial P}{\partial y}(x,f(x))\ne 0$ for $x\in U$, then in Theorem \ref{maintwr} we obtain
$$
\varrho_a(f)\leq 1-\frac{1}{d(3d-2)^{n}+1}, \quad\hbox{provided}\quad d\ge 2.
$$
The above estimates are comparable with the Solern\'o estimate \eqref{eqsolerno}, but our estimates are explicit.

As a corollary, we obtain the following inequality (see Corollary \ref{wn012}):
\begin{equation}\label{eq052}
|\nabla f(x)|\ge C \dist (x,f^{-1}(0))^{2(2d-1)^{3n+1}-1}\quad \hbox{in a neighbourhood of $a$}.
\end{equation}
If additionally $n\ge 2$ and  $\frac{\partial P}{\partial y}(x,f(x))\ne 0$ for $x\in U$, then 
\begin{equation}\label{eq051}
|\nabla f(x)|\ge C \dist (x,f^{-1}(0))^{d(3d-2)^{n}}\quad \hbox{in a neighbourhood of $a$}.
\end{equation}
The inequalities  \eqref{eq052}, \eqref{eq051} are essential points in the effective estimate of the degree of sufficiency of non-isolated Nash function singularities given in the next section. The proof of these inequalities is based on Theorem \ref{maintwrIII} and estimates of the length of trajectories of the vector field $\nabla f$ in $U\setminus f^{-1}(0)$ (see Theorem \ref{lengthtrajectory}).

\subsection{Sufficiency of non-isolated Nash function singularities}
 

Let $\C^k_a(n)$  denote the set of $\C^k$ real functions defined in neighbourhoods of $a\in \rr^n$. 

By a $k$-\emph{jet} at $a\in \rr^n$ in the class  $\C^\ell$ we mean a family of functions $w\subset \C^\ell_a(n)$, called \emph{$\C^\ell$-realizations} of this jet, possessing the same Taylor polynomial of degree $k$ at $a$. We also say that $f$ \emph{determines} a $k$-jet at $a$ in  $\C^\ell$  if $f$ is a $\C^\ell$-realization of this jet. 
For a function $f\in \C^k_a(n)$, we denote by $j^kf(a)$ the $k$-jet at $a$ (in  $\C^k$) determined by $f$. 

Let $Z\subset \rr^n$ be a set such that $0\in Z$ and let $k\in\zz$, $k>0$. By a $k$-$Z$-\emph{jet in the class $\C^k$}, 
or briefly a $k$-$Z$-\emph{jet}, we mean an equivalence class $w\subset \C^k_0(n)$ of the following equivalence relation: $f\sim g$ iff for some neighbourhood $U\subset \rr^n$ of the origin, $j^kf(a)=j^kg(a)$ for $a\in Z\cap U$ (cf. \cite{MigusRS}, \cite{Xu}). The functions $f\in w$ are called \emph{$\C^k$-$Z$-realizations} of the jet $w$ and we write $w=j^{k}_{Z}f$. The set of all jets $j^{k}_{Z}f$ is denoted by $J^{k}_Z(n)$.

The $k$-$Z$-jet $w\in J^{k}_Z(n)$ is said to be $\C^r$-$Z$-\emph{sufficient} (resp. $Z$-$v$-\emph{sufficient}) in the class  $\C^k$ if for every of its $\C^k$-$Z$-realizations $f$ and $g$ there exist sufficiently small neighbourhoods $U_1,\,U_2\subset\rr^n$ of $0$, and  a $\C^r$ diffeomorphism $\vf:U_1\to U_2$, such that $f\circ \vf=g$ in $U_1$ (resp. there exists a homeomorphism $\vf: [f^{-1}(0)\cup Z]\cap U_1 \to[g^{-1}(0)\cup Z]\cap U_2$ with  $\vf(0)=0$ and $\vf(Z\cap U_1)= Z\cap U_2$).

The classical and significant result on sufficiency of jets is the following:

\begin{twr}[Kuiper, Kuo, Bochnak-{\L}ojasiewicz]\label{KKBL}
Let $w$ be a $k$-jet at $0\in \rr^n$ and let $f$ be its $\C^k$-realization. If $f(0)=0$  then the following conditions are equivalent:
\begin{itemize}
\item[(a)] $w$ is $\C^0$-sufficient in $\C^k$,
\item[(b)] $w$ is $v$-sufficient in $\C^k$,
\item[(c)] $|\nabla f(x)|\ge C|x|^{k-1}$ in a neighbourhood of the origin for some $C>0$.
\end{itemize}
\end{twr}

The implication (c)$\Rightarrow$(a) was proved by N. H.~Kuiper \cite{Kui} and T. C.~Kuo \cite{Kuo}, (b)$\Rightarrow$(c)  by J.~Bochnak and S. \L ojasiewicz \cite{BL}, and  (a)$\Rightarrow$(b) is obvious (cf. \cite{OSS}). 

Let us recall the notions of isotopy and topological triviality. 
  Let  $\Omega\subset\rr^n$ be a neighbourhood of $0\in\rr^n$ and let $Z\subset \rr^n$ with $0\in Z$. 
 
A continuous mapping $H\colon\Omega\times [0,1]\to\rr^n$ is called an \emph{isotopy near $Z$ at zero} if:

\noindent{\rm (a)} $H_0(x)=x$ for $x\in \Omega$ and $H_t(x)=x$ for $t\in[0,1]$ and $x\in \Omega\cap Z$,\\
{\rm (b)} for any $t$ the mapping  $H_t:\Omega\to \rr^n$ is a homeomorphism onto $H_t(\Omega)$,\\
where   $H_t(x)=H(x,t)$ for $x\in\Omega$, $t\in [0,1]$.

Functions $f:\Omega_1\to\rr $, $g:\Omega_2\to \rr$, where $\Omega_1,\Omega_2\subset\rr^n$ are neighbourhoods of $0\in\rr^n$, are  
called \emph{isotopical near $Z$ at zero} if there exists an isotopy near $Z$ at zero, 
  $H:\Omega\times [0,1]\to\rr^n$, with $\Omega\subset\Omega_1\cap \Omega_2$, such that
  $f(H_1(x))=g(x)$, $x\in\Omega$. 

A deformation $f+tg$ is called \emph{topologically trivial near $Z$} along $[0,1]$ if there exists an isotopy near $Z$ at zero, $H:\Omega \times [0,1]\to \rr^n$, with $\Omega \subset \Omega_1\cap\Omega_2$, such that $f(H(t,x))+tg(H(t,x))$ does not depend on $t$.

%
Theorem \ref{KKBL} concerns the case of an isolated singularity of $f$ at $0$, i.e. $0$ is an isolated zero of $\nabla f$. In the case of a non-isolated singularity of $f$ at $0$, from   \cite[Theorems 1.3 and 1.4]{MigusRS} (cf. \cite{Xu}) we have the following   criterion 
for sufficiency of jets.

\begin{twr}\label{XuXu}
Let $f\in \C_0^k(n)$ be a $\C^k$-$Z$-realization of a $k$-$Z$-jet $w\in J^{k}_Z(n)$, where $k>1$ and $Z=f^{-1}(0)$, $0\in Z$, and suppose $(\nabla f)^{-1}(0)\subset Z$. Then the following conditions are equivalent:

\noindent{\rm (a)}  The $k$-$Z$-jet $w$ is $\C^0$-$Z$-sufficient in $\C^k$. 


\noindent{\rm (b)} For any $\C^k$-$Z$-realizations $f_1,\,f_2$ of $w$,  the deformation $f_1+t(f_2-f_1)$, $t\in\rr$, is topologically trivial along $[0,1]$.

\noindent{\rm (c)} Any two $\C^k$-$Z$-realizations 
 of $w$ are  isotopical at zero. 

\noindent{\rm(d)} The $k$-$Z$-jet $w$ is $Z$-$v$-sufficient in $\C^k$.

\noindent{\rm(e)} There exists a positive constant $C$ such that
\begin{equation*}\label{KKcondition}
|\nabla f(x)|\ge C\dist(x,Z)^{k-1}\quad \hbox{in a neighbourhood of the origin}.
\end{equation*}
\end{twr}

 
  Let $f:U\to \rr$ be a Nash function, where $U\subset \rr^n$ is a neighbourhood of the origin, let $Z=f^{-1}(0)$, and suppose  $0\in Z$.
 
 The main result of this paper is the following corollary from Theorem \ref{XuXu} and inequality \eqref{eq052}.
 
 \begin{twr}\label{mainsuffjetLojineq}
Let  $k=2(2d-1)^{3n+1}$, where $d=\deg_0 f$, and let $w\in J^{k}_Z(n)$ be the $k$-$Z$-jet for which $f$ is a $\C^k$-$Z$-realization. Then:
 
 \noindent{\rm(a)} The $k$-$Z$-jet $w$ is  $\C^0$-$Z$-sufficient in $\C^k$.
 
 \noindent{\rm(b)}  For any $\C^k$-$Z$-realizations $f_1,\,f_2$ of $w$,  the deformation $f_1+t(f_2-f_1)$, $t\in\rr$, is topologically trivial along $[0,1]$. 
 
  \noindent{\rm(c)}  Any two $\C^k$-$Z$-realizations 
 of $w$ are  isotopical at zero. 
 
  \noindent{\rm(d)} The $k$-$Z$-jet $w$ is $Z$-$v$-sufficient in $\C^k$. 
  \end{twr}

Under additional assumption on $f$, from Theorem \ref{XuXu} and inequality \eqref{eq051}, we obtain 

\begin{twr}\label{mainsuffjetLojineqII}
 Assume that there exists a nonzero polynomial $P\in\rr[x,y]$ such that $P(x,f(x))=0$ and $\frac{\partial P}{\partial y}(x,f(x))\ne 0$ for $x\in U$. Then the assertion of Theorem \ref{mainsuffjetLojineq} holds with $k=d(3d-2)^{n}+1$, where $d=\deg P$.
  \end{twr}
 
 \begin{remark}\label{rem1}
 If $f$ is a polynomial of degree $d>1$ or a rational function $f={p}\slash{q}$, where $p(0)=0$, $q(0)\ne 0$ and $d=\deg p$, then from Theorem \ref{XuXu} and by \eqref{DKineq}, the assertion of Theorem \ref{mainsuffjetLojineq} holds with $k=d(3d-3)^{n-1}$. If additionally the origin is an isolated zero of $f$, then by \eqref{Gw2} the assertion of Theorem \ref{mainsuffjetLojineq} holds with $k=(d-1)^n+1$.
 \end{remark}
 

 
\section{{\L}ojasiewicz gradient inequality}\label{Lojgradineqsect}

Let $f:U \to \rr$, where $U\subset \rr^n$ is a connected neighbourhood of $a\in\rr^n$, be a Nash function. 
Let $P\in\rr[x,y]$ be the unique irreducible real polynomial such that 
\begin{equation}\label{eqP}
P(x,f(x))=0\quad \hbox{for }x\in U,
\end{equation}
and let 
$$
d=\deg P.
$$
 We will call this number $d$ the  \emph{degree of the Nash function} $f$ \emph{at} $a$ and denote it by $\deg_a f$. Obviously $d=\deg_a f > 0$ is uniquely determined. For $d=1$, the function $f$ is linear and \eqref{Lojineq1}  holds with $\varrho=0$, so we will assume that $d>1$. We will also assume that $\nabla f(a)=0$, because in the opposite case  \eqref{Lojineq1}  holds with $\varrho=0$. 

Put 
$$
\mathcal{R}(n,d)=\max\{2d(2d-1),
d(3d-2)^n\}+1. 
$$

The main result of this section is the following theorem.

\begin{twr}\label{maintwr}
Let $f: U\to \rr$ be a nonzero Nash function such that $f(a)=0$ and $\nabla f(a)=0$. Assume that for the unique polynomial $P$ satisfying \eqref{eqP} we have
\begin{equation}\label{geberalassumption}
\frac{\partial P}{\partial y}(x,f(x))\ne 0\quad\hbox{for }x\in U.
\end{equation}
Then $\varrho_a(f)\le 1-\frac{1}{\mathcal{R}(n,d)}$. Moreover, for $\varrho=1-\frac{1}{\mathcal{R}(n,d)}$ and some constants $C,\varepsilon>0$,
\begin{equation}\label{Lojineqmain}
|\nabla f(x)|\ge C|f(x)|^\varrho\quad\hbox{for}\quad |x-a|<\varepsilon,\quad |f(x)|<\varepsilon.
\end{equation}
\end{twr}

Without the assumption \eqref{geberalassumption}, we have a somewhat weaker estimation of the exponent $\varrho_a(f)$ than in Theorem \ref{maintwr}. Namely, let
$$
\mathcal{S}(n,d)=
2(2d-1)^{3n+1}.
$$

\begin{twr}\label{maintwrIII}
Let $f: U\to \rr$ be a nonzero Nash function such that $f(a)=0$ and $\nabla f(a)=0$ and let $P$ be  the unique polynomial satisfying \eqref{eqP}.  
Then $\varrho_a(f)\le 1-\frac{1}{\mathcal{S}(n,d)}$. Moreover, \eqref{Lojineqmain} holds actually 
with $\varrho=1-\frac{1}{\mathcal{S}(n,d)}$.
\end{twr}

Theorems \ref{maintwr} and  \ref{maintwrIII} are generalizations for Nash functions of the above mentioned  results by J. Gwo\'zdziewicz and  D.~D'Acunto and K.~Kurdyka in the polynomial function case. 
They are also comparable with  Solern\'o's estimate \eqref{eqsolerno}, but our estimates are explicit. In the case of Nash functions with isolated singularity at zero, a similar result was obtained in \cite{KOSS}. 


We give the proofs of Theorems  \ref{maintwr} and  \ref{maintwrIII}  in Section \ref{roz2}.

\section{{\L}ojasiewicz inequality}

Let $X\subset \rr^n$ be a compact semialgebraic set and let $f:X\to\rr$ be a Nash function. Then $f$ is defined in a neighbourhood of $X$. So, there exists a compact semialgebraic set $Y\subset \rr^n$ such that $X\subset \Int Y$ and $f$ is defined on $Y$. 

The \emph{degree} of $f$  is defined to be  $\sup\{\deg_a f:a\in X\}$ and is denoted by $\deg_X f$. In fact, $\deg_X f=\max\{\deg_a f:a\in X\}$. Moreover, one can assume that $Y$ was chosen in such a manner that $\deg _X f=\deg_Y f$.
  
Let $\dist (x,V)$  denote the distance of a point $x\in\er^n$ to a set $V\subset \rr^n$ in the Euclidean norm (with  $\dist (x,V)=1$ if $V=\emptyset$).

\subsection{Global gradient Łojasiewicz inequality}

Theorems \ref{maintwr} and \ref{maintwrIII} have a local character. From these theorems  we obtain a \emph{global {\L}ojasiewicz gradient inequality}.

\begin{cor}\label{corLojgradglobal}
Let $d=\deg_X f$. If  $(\nabla f)^{-1}(0)\subset f^{-1}(0)$ then for some positive constant $C$, 
\begin{equation}\label{eqLojgradglobal}
|\nabla f(x)|\ge C|f(x)|^\varrho\quad\hbox{for}\quad x\in X
\end{equation}
with $\varrho=1-\frac{1}{\mathcal{S}(n,d)}$. If additionally there exists a polynomial $P\in\rr[x,y]$ such that $P(x,f(x))=0$ and $\frac{\partial P}{\partial y}(x,f(x))\ne 0$ for $x\in X$ and $d_1=\deg P$, then \eqref{eqLojgradglobal} holds with $\varrho=1-\frac{1}{\mathcal{R}(n,d_1)}$.
\end{cor}

Denote by $\varrho_X(f)$ the smallest exponent $\varrho$ for which \eqref{eqLojgradglobal} holds. We call it the \emph{{\L}ojasiewicz exponent in the gradient inequality} on $X$. It is known that the inequality \eqref{eqLojgradglobal} holds with $\varrho=\varrho_X(f)$. So, from Corollary \ref{corLojgradglobal} we obtain

\begin{cor}\label{corestLojexp}
  $\varrho_X(f)\leq 1-\frac{1}{S(n,d)}$.
\end{cor}
\subsection{Length of trajectory}\label{lenghtsect}

Let $f:X\to\rr$ be 
a nonzero Nash function such that $(\nabla f)^{-1}(0)\subset f^{-1}(0)$, let $\varrho\in(0,1)$ and $C>0$ be such that the global inequality \eqref{eqLojgradglobal} in Corollary \ref{corLojgradglobal}  holds in $X$, 
and let $V=f^{-1}(0)$. Then $\nabla f(x)\ne 0$ for $x\in X\setminus V$.

Let $\varphi(t)=|t|^{1-\varrho}$ for $t\in \er$. By the same argument as in the proof of \cite[Proposition 1]{KS1} we obtain  (cf. \cite{KMP})

\begin{prop}[Kurdyka-{\L}ojasiewicz inequality]\label{KurdykaLojasiewicz}
Under the above notations, 
\begin{equation*}  
|\nabla (\varphi\circ f)(x)|\ge (1-\varrho)C \quad\hbox{for $x\in X\setminus V$}. 
\end{equation*}
\end{prop}

We will also  assume that $\overline{\Int X\setminus V}=X$. Let 
$$
U_{X,f}=\left\{x\in \Int X:
\frac{1}{C(1-\varrho)}|f(x)|^{1-\varrho}
<\dist (x,\rr^n\setminus X) \right\}.
$$
Then $U_{X,f}\subset X$ is a neighbourhood of $(\Int X)\cap V$.

Take a global trajectory $\gamma:[0,s)\to U_{X,f}\setminus V$ of the  vector field
\begin{equation*}\label{eq3H}
H(x)=-\sign f(x)\frac{\nabla f(x)}{|\nabla f(x)|} \quad\hbox{for $x\in U_{X,f}\setminus V$}. 
\end{equation*}
Then the function $f\circ \gamma$ is  monotonic, 
 so  the limit $\lim_{t\to s}f\circ \gamma(t)$ exists.

Let $\length \gamma$ denote the length of $\gamma$. Since $|\gamma'(t)|=1$, we have  $\length \gamma=s$.

The following generalization of \cite[Theorem 1]{KS1} has a similar proof.

\begin{twr}\label{lengthtrajectory}  The limit $\lim_{t\to s}\gamma(t)$ exists and belongs to $V$. Moreover,
\begin{equation*}
\dist (\gamma(0),V)\le\length \gamma\le \frac{1}{(1-\varrho)C}|f(\gamma(0))|^{1-\varrho}. 
\end{equation*}
\end{twr}

\begin{proof}
 Let $s_1\in [0,s)$ and $\gamma_{s_1}=\gamma|_{[0,s_1]}$. Then $\length \gamma_{s_1}=s_1$. 
Since $\frac{\nabla f}{|\nabla f|}\sign f(x)=\frac{\nabla (\varphi \circ f)}{|\nabla(\varphi \circ f)|}$ for $x\in U\setminus V$, 
it follows that
\begin{equation*}
(\varphi\circ f\circ \gamma)'=\langle \nabla (\varphi \circ f)\circ \gamma,\gamma '\rangle =-|\nabla (\varphi \circ f)\circ \gamma|,
\end{equation*}
where $\langle \cdot,\cdot\rangle$ denotes the standard scalar product in
$\er^n$, and Proposition \ref{KurdykaLojasiewicz} gives
\begin{equation*}
\begin{split}
\varphi(f(\gamma(0)))-\varphi(f(\gamma(s_1)))&=-s_1(\varphi\circ f\circ \gamma)'(t)=s_1 |\nabla (\varphi\circ f)\circ \gamma(t)|\\
&\ge 
(1-\varrho)C\length \gamma_{s_1}
\end{split}
\end{equation*}
for some $t\in [0,s_1]$. Then, letting $s_1\to s$, from the definition of $\vf$ we have
$$
\length \gamma \le \frac{1}{(1-\varrho)C}(|f(\gamma(0))|^{1-\varrho}-\alpha)\le \frac{1}{(1-\varrho)C}|f(\gamma(0))|^{1-\varrho},
$$
where $\alpha=\lim_{s_1\to s}|f(\gamma(s_1))|^{1-\varrho}\ge 0$. 

Since $\gamma(0)\in U_{X,f}$, we see that $\length\gamma< \dist(\gamma(0),\rr^n\setminus X)$, so the limit $\lim_{t\to s}\gamma(t)$ certainly exists and belongs to $U_{X,f}$. Consequently,  $\lim_{t\to s}\gamma(t)\in V$ and $\length \gamma\ge \dist (\gamma(0),V)$. This  gives the assertion.
\end{proof}

From Theorem \ref{lengthtrajectory} 
 we have

\begin{cor}\label{wn01} Under the assumptions and notations of Theorem \ref{lengthtrajectory},
\begin{equation*}\label{eq-15}
|f(x)|\ge \left(C(1-\varrho)\right)^{1\slash(1-\varrho)} \dist (x,V)^{1\slash(1-\varrho)}, \quad x\in U_{X,f},
\end{equation*}
 and
\begin{equation*}\label{eq05}
|\nabla f(x)|\ge \left(C(1-\varrho)\right)^{\varrho\slash(1-\varrho)} \dist (x,V)^{\varrho\slash(1-\varrho)},\quad x\in U_{X,f}.
\end{equation*}
\end{cor}

Similarly to \cite{KS1}, we obtain a version of the above corollary in the complex case with the same formulation.

From Corollaries \ref{corLojgradglobal}, \ref{wn01} and Theorem \ref{maintwrIII}, we immediately obtain

\begin{cor}\label{wn012} Let $d=\deg_X f$. Then there exists a positive constant $C$ such that 
\begin{equation*}\label{eq-15}
|f(x)|\ge C \dist (x,V)^{2(2d-1)^{3n+1}}, \quad x\in X,
\end{equation*}
 and
\begin{equation*}\label{eq05}
|\nabla f(x)|\ge C \dist (x,V)^{2(2d-1)^{3n+1}-1},\quad x\in X.
\end{equation*}

If additionally $n\ge 2$ and  there exists a polynomial $P\in\rr[x,y]$ such that $P(x,f(x))=0$ and $\frac{\partial P}{\partial y}(x,f(x))\ne 0$ for $x\in X$, and $d=\deg P$, then
\begin{equation*}\label{eq-15}
|f(x)|\ge C \dist (x,V)^{d(3d-2)^{n}+1}, \quad x\in X,
\end{equation*}
 and
\begin{equation*}\label{eq05}
|\nabla f(x)|\ge C \dist (x,V)^{d(3d-2)^{n}},\quad x\in X.
\end{equation*}
\end{cor}

\begin{proof}
Take a compact semialgebraic set $Y\subset \rr^n$ such that $X\subset \Int Y$ and $Y\subset \{x\in\rr^n:\dist(x,X)<\ve\}$. If $\ve$ is sufficiently small, then we can consider the function $f$ on $Y$. Then we may assume that $\deg _Y f=\deg_X f$ and $(\nabla f)^{-1}(0)\subset f^{-1}(0)$ after extending  $f$ 
onto~$Y$. So, the assertions of Theorem \ref{lengthtrajectory} and Corollary \ref{wn01} hold with $\varrho =1-\frac{1}{S(n,d)}$ on the set $U_{Y,f}$. Hence the assertions hold for $x\in X\cap U_{Y,f}$. By the definition of $U_{Y,f}$, we see that $X\setminus U_{Y,f}$ is a compact set and $\min \{|x-y|:x\in V,\;y\in X\setminus U_{Y,f}\}>0$. So, diminishing $C$ if necessary, we obtain the first part of the assertion. The second part is proved analogously. 
\end{proof}

\subsection{{\L}ojasiewicz exponent} 
 Corollary \ref{wn01} implies the known fact that the exponents $\alpha>0$ in the inequality
\begin{equation}\label{eq-15}
|f(x)|\ge C\dist (x,V)^\alpha,\quad x\in X,
\end{equation}
for some positive constant $C$, are bounded below. The inequality \eqref{eq-15} is called the \emph{{\L}ojasiewicz inequality for $f$ on} $X$ and the lower bound of the exponents $\alpha>0$ is the \emph{{\L}ojasiewicz exponent} of $f$ on $X$,  denoted by $\wll_X(f)$. It is known that   \eqref{eq-15} holds with $\alpha=\wll_X(f)$ and some positive constant $C$. 

From  Theorem \ref{lengthtrajectory} we obtain

\begin{cor}\label{corrhoLojexp}
$\wll_X(f)\leq \frac{1}{1-\varrho_X(f)}$. 
\end{cor}

 Corollary \ref{wn01} implies

\begin{cor}\label{corLojexp} If $d=\deg_X f$, then 
$\wll_X(f)\leq 2(2d-1)^{3n+1}$.
\end{cor}

For $n\ge 4$ the above estimate is  sharper than the one given in \cite{KSS} for continuous semialgebraic functions: $\wll_X(f)\leq d(6d-3)^{n+r-1}$, where $r\leq\frac{n(n+1)}{2}$ is the degree of complexity of $f$, equal to the number of inequalities  necessary to define the graph of $f$, and $d$ is the maximal degree of polynomials describing the graph of~$f$. Consequently, this gives the estimate 
$
\wll_X(f)\leq d(6d-3)^{n+n(n+1)/2-1}
$  
in terms of the degree only. So, the estimate in Corollary \ref{corLojexp} is more exact than  the one above for $n\ge 4$.

\section{Total degree of algebraic sets}

Let $\cc[x]$  denote the ring of complex polynomials in  $x=(x_1,\ldots,x_n)$.


Let $f=(f_1,\ldots,f_r):\cc^n\to\cc^r$ be a polynomial mapping with $\deg f_i >0$ for $i=1,\ldots,r$. Let $V=f^{-1}(0)\subset \cc^n$. 

The \emph{total degree} of $V$ is the number
$$
\delta (V)=\deg V_1+\cdots+\deg V_s,
$$
where $V=V_1\cup\cdots\cup V_s$ is the decomposition  into irreducible components (see \cite{Lo3}).

We have the following useful fact (see \cite{Lo3}).

\begin{fact}\label{deltainters}
If $V,W\subset \cc^n$ are algebraic sets, then
$$
\delta(V\cap W)\le \delta(V)\delta(W).
$$
\end{fact}

From Fact \ref{deltainters} and the definition of total degree of algebraic sets we have the following two facts (cf. \cite{Lo3}).

\begin{fact}\label{Fact1}$\delta(V)\le \deg f_1\cdots \deg f_r$. 
In particular, for any irreducible component $V_j$ of $V$ we have 
$$
\deg V_j\le \deg f_1\cdots \deg f_r.
$$
\end{fact}

\begin{fact}\label{Fact2} Let $L:\cc^n\to\cc^k$ be a linear mapping. Then 
$$
\delta(\overline{L(V)})\le \delta(V).
$$
\end{fact}

We will need the following lemma (see \cite[Lemma 3.20]{KOSS}).

\begin{lem}\label{Fact3} Let $V_j$ be an irreducible component of the set $V$, and suppose $\dim V_j\ge 1$. Then for a generic linear mapping $L=(L_1,\ldots,L_{n-1}):\cc^r\to\cc^{n-1}$ the set $V_j$ is an irreducible component of the set of common zeros of the  system of equations
$$
L_i\circ f=0,\quad i=1,\ldots, n-1.
$$
In particular,
$$
\deg V_j\le \deg (L_1\circ f)\cdots\deg (L_{n-1}\circ f).
$$
Moreover, we can take $L_1(y_1,\ldots,y_r)=y_1$.
\end{lem}

\section{Proofs of Theorems \ref{maintwr} and  \ref{maintwrIII}}\label{roz2}

The idea of the proofs is similar to that in    \cite[proof of Theorem 1.2]{KOSS}.

Without loss of generality, we may assume that $a=0$. 
Let $f:U \to \rr$ be a nonzero Nash function defined in an open neighbourhood $U\subset \rr^n$ of the origin such that $f(0)=0$ and $\nabla f(0)=0$. Let $P\in\rr[x,y]$ be the unique irreducible polynomial satisfying \eqref{eqP} and let $d=\deg P$. 

Since the set of critical values of a  differentiable semialgebraic function is finite, we have

\begin{fact}\label{Fact4} There exists $\varepsilon>0$ such that $f$ has no critical values in the interval $(-\varepsilon,\varepsilon)$ except   $0$.
\end{fact}

Let $\varepsilon>0$ be as in Fact \ref{Fact4}. 
Take $r>0$. Denote by $\Omega$ the closed ball 
$$
\Omega:=\{x\in\rr^n:|x|\le r\}
$$
and by $\partial \Omega$  the sphere $\{x\in\rr^n:|x|=r\}$. 
Suppose that 
$\Omega\subset U$.  
Define  a semialgebraic set $\Gamma\subset \Omega$ by
$$
\Gamma:=\{x\in \Omega:\forall_{\zeta\in \Omega}\;f(x)=f(\zeta)\;\Rightarrow\;|\nabla f(x)|\le|\nabla f(\zeta)|\}.
$$
Then by the definition of $\Gamma$ we have

\begin{fact}\label{Factproof1}
Let $\varrho\in\rr$ and let $C>0$. If $|\nabla f(x)|\ge C|f(x)|^\varrho$ for $x\in \Gamma$ such that $|f(x)|<\varepsilon$, then 
$|\nabla f(x)|\ge C|f(x)|^\varrho$ for $x\in\Omega$, $|f(x)|<\varepsilon$.
\end{fact}

Let $\varrho_0=\varrho_0(f)$. Then, decreasing $r$ if necessary, we can assume that
\begin{equation}\label{Lojgradinequality}
|\nabla f(x)|\ge C|f(x)|^{\varrho_0} \quad  \hbox{for $x\in\Omega$ and some constant $C>0$}.
\end{equation}
Let us fix such an $r$.

Consider the case $n=1$. Denote by $\ord_0 f$ the order of $f$ at zero. 
Then $f$ has an isolated zero and singularity at zero, $\ord_0 f>0$ and the inequality \eqref{Lojineqmain} holds with 
\begin{equation}\label{caen1}
\varrho_0(f)=\frac{\ord_0 f-1}{\ord_0 f}=1-\frac{1}{\ord_0 f}.
\end{equation}
Let the polynomial $P$ be of the form $
P(x_1,y)=p_0(x_1)y^d+p_1(x_1)y^{d-1}+\cdots +p_d(x_1)$, 
where $p_0,\ldots,p_d\in\rr[x_1]$. As $P$ is irreducible, $p_d\ne 0$ and $\ord_0 p_d\le d$. Since
$$
-p_d(x_1)=f(x_1)(p_0(x_1)(f(x_1))^{d-1}+p_1(x_1)(f(x_1))^{d-2}+\cdots +p_{d-1}(x_1)),
$$
we have $\ord_0 f\le \ord_0 p_d\le d$. Together  with \eqref{caen1} this gives \eqref{Lojineqmain} with $\varrho_0(f)=1-\frac{1}{d}$ and the assertions of Theorems \ref{maintwr} and  \ref{maintwrIII}  in the case $n=1$. 

In the remainder of this article we will assume that $n>1$.

By \eqref{Lojgradinequality} and the Curve Selection Lemma, there exists an analytic curve  $\varphi:[0,1)\to \Omega$ for which $f(\varphi(0))=0$, $f(\varphi(\xi))\ne 0$ for $\xi\in(0,1)$ and for some constant $C_{1}>0$,
\begin{equation}\label{Lojgradinequality2}
C|f(\varphi(\xi))|^{\varrho_0}\le |\nabla f(\varphi(\xi))|\le C_{1}|f(\varphi(\xi))|^{\varrho_0},\quad \xi\in [0,1)
\end{equation}
(cf. \cite{Sp1}). 
By Fact \ref{Factproof1} we may assume that $\varphi ([0,1))\subset \Gamma$. Then we have two cases: 

I. $\varphi\big(( 0,1)\big)\subset\operatorname{Int}\Omega$,

II. $\varphi\big([ 0,1)\big)\subset\partial\Omega$.

We will use the Lagrange multipliers theorem to describe        the relation between the values $y=f(x)$ and $u=|\nabla f(x)|^2$ for $x\in\Gamma$, so we put
$$
\Gamma_I =\{x\in\Omega:\exists_{\la\in\rr}\, \nabla|\nabla f(x)|^2-\la\nabla f(x)=0\},
$$
$$
{\Gamma}_{II} =\{x\in\partial \Omega:|f(x)|<\varepsilon\;\land\;\exists_{\la_1,\la_2\in\rr}\, \nabla|\nabla f(x)|^2-\la_1\nabla f(x)-2\la_2x=0\}.
$$

To fulfill the assumptions of the Lagrange theorem we will need 
\begin{lem}\label{ggg}
There exists $\ve>0$ such that for every $x\in \partial\Omega$ and every $y\in\rr$ such that $0<|y|<\ve$ and $y=f(x)$, the vectors $
\nabla \big(|x|^2-r^2\big)$ and  $\nabla f(x)$  (that is,  $2x$ and  $\nabla f(x)$) are linearly independent.
\end{lem}

\begin{proof}
If $f|_{\partial \Omega}$ is a constant function then the assertion is obvious. Assume that $f$ is not  constant  on $\partial \Omega$. Then, by Fact \ref{Fact4}, there exists $\ve>0$ such that $\nabla f(x)\ne 0$ for $x\in\partial \Omega$, $0<|f(x)|<\ve$. 

Suppose to the contrary that for any $\ve>0$ there exist $x\in\partial\Omega$ and $y_{\ve}\in\rr$ with $0<|y_{\ve}|<\ve$ such that $y_{\ve}=f(x)$ and  
$\nabla f(x)=\xi\cdot 2x$ for some $\xi\in\rr\setminus\{0\}$. 
Then  by the Curve Selection Lemma there exist analytic curves $\gamma:[ 0,1)\to \partial\Omega$ with $\gamma((0,1))\subset\Omega\setminus f^{-1}(0)$ and $f(\gamma(0))=0$, and $\alpha:[ 0,1)\to\rr $, 
 such that for $t\in (0,1)$,
\begin{equation*}\label{gx}
\nabla f\big(\gamma(t)\big)=\alpha(t)\cdot 2\gamma(t).
\end{equation*}
Then 
$$
(f\circ \gamma)'(t)=\langle \nabla f(\gamma(t)),\gamma'(t)\rangle=\alpha(t)\langle \gamma(t),\gamma'(t)\rangle=0,
$$
and consequently $f\circ \gamma$ is a constant function equal to $0$. This contradicts the choice of $\gamma$ and ends the proof.
\end{proof}

By the Lagrange multipliers theorem,  Fact \ref{Fact4} and Lemma \ref{ggg} we obtain 

\begin{fact}\label{fact1IandIIl}
Let $\varepsilon>0$ fulfill Fact \ref{Fact4} and Lemma \ref{ggg}. Take a point $x_0\in \Omega$ such that $0<|f(x_0)|<\ve$.

{\rm (a)} If $x_0\in\Gamma\cap\Int \Omega$  then $x_0$ is a 
lower critical point of the function $\Omega\ni x\mapsto |\nabla f(x)|^2\in\rr$ on the set $f^{-1}(f(x_0))\cap \Omega$. In particular, $\Gamma\cap \Int \Omega\subset \Gamma_I$.

{\rm (b)} If $n\ge 3$, $x_0\in\Gamma\cap \partial \Omega$ then $x_0$ is a lower critical point of the function $\partial \Omega\ni x\mapsto |\nabla f(x)|^2\in\rr$ on the set $f^{-1}(f(x_0))\cap \partial \Omega$. In particular, $\Gamma \cap \partial \Omega\subset \Gamma_{II}$.
\end{fact}

Let $\M=\cc^n\times\cc\times\cc\times\cc^n\times\cc^n$, and let $\X\subset \M$ be the Zariski closure of the set
$$
\{(x,f(x),|\nabla f(x)|^2,\nabla f(x),\nabla |\nabla f(x)|^2)\in \M:x\in\Omega\}.
$$


We will determine polynomials describing a certain algebraic set $\Y\subset \M$ containing $\X$ as an irreducible component. Let  $G\in \cc[x,y,u]$, where $u$ is a variable, be the polynomial defined by
\begin{equation}\label{eqdefG}
G(x,y,u)=\sum_{i=1}^n\left(\frac{\partial P}{\partial x_i}(x,y)\right)^2-\left(\frac{\partial P}{\partial y}(x,y)\right)^2 \cdot u.
\end{equation}
It is easy to observe that 
$G(x,f(x),|\nabla f(x)|^2)=0$ for $x\in \Omega$. In particular, the polynomial $G$ vanishes on $\X$. 

Take systems of variables $t=(t_1,\ldots,t_n)$, $z=(z_1,\ldots,z_n)$, and let 
 $G_1,G_{2,i},G_{3,i} \in\cc[x,y,u,t,z]$ be  defined by
\begin{align}
G_1(x,y,u)&=u-t_1^2-\cdots- t_n^2, \,,\nonumber&\\
G_{2,i}(x,y,t)&=\frac{\partial P}{\partial x_i}(x,y)+\frac{\partial P}{\partial y}(x,y)t_i \,,& 1\le i\le n,\nonumber\\
G_{3,i}(x,y,u,t,z)&=\frac{\partial G}{\partial x_i}\left(x,y,u\right)+\frac{\partial G}{\partial y}\left(x,y,u\right)t_i \nonumber
&\\
&\qquad\qquad\qquad-\left(\frac{\partial P}{\partial y}\left(x,y\right)\right)^2\cdot z_i\, ,& 1\le i\le n.\nonumber
\end{align}

Let $\Y\subset \M$ be the closure of the  
constructible set
\begin{multline*}
\Y^0=\{w=(x,y,u,t,z)\in \M: P(x,y)=0,\;\frac{\partial P}{\partial y}(x,y)\ne 0,\; G_1(x,y,u)=0,\\
G_{2,i}(x,y,t)=0,\; G_{3,i}(w)=0,\;1\le i\le n\}.
\end{multline*}
Obviously $\X\subset \Y$, and locally  $\Y^0$ is the graph of a complex Nash mapping (i.e., a holomorphic mapping with semialgebraic graph). Moreover, we have

\begin{lem}\label{factY0smooth}
The set $\X$ is an irreducible component of $\,\Y$. Moreover, $\Y^0$ is a Zariski open and dense subset of $\Y$, and any point $w=(x_0,y_0,u_0,t_0,z_0)\in\Y^0$ has a neighbourhood $B\subset \M$ such that $\Y\cap B=\Y^0\cap B$ and
\begin{equation*}\label{eqformY0graph}
\Y^0\cap B=\left\{w=\left(x,g(x),h(x),\nabla g(x),\nabla h(x)\right)\in\M:x\in \Delta\right\}
\end{equation*}
for some holomorphic  function $g:\Delta\to \cc$, where $\Delta\subset \cc^n$ is a neighbourhood of $x_0$, and $h(x)=\left(\frac{\partial g}{\partial x_1}(x)\right)^2+\cdots+\left(\frac{\partial g}{\partial x_n}(x)\right)^2$.
 \end{lem} 

\begin{proof}
Since $P$ is an irreducible polynomial,  $\frac{\partial P}{\partial y}$ does not vanish on $\X$. So, by the Implicit Function Theorem, 
$
\{w=(x,y,u,t,z)\in \X:\frac{\partial P}{\partial y}(x,y)\ne0\}
$ 
is an open and dense subset of $\X$, and moreover it is a smooth and connected submanifold of~$\Y^0$. Consequently, $\X$ is an irreducible component of $\Y$. The ``moreover'' part of the assertion follows  immediately  from  the Implicit Function Theorem.
\end{proof}



Define $G_0,G_{4,i,j},G_{4,i,j,k}\in\cc[x,y,u,t,z]$  by
\begin{align}
G_0(x)&=x_1^2+\cdots+x_n^2-r^2,\nonumber\\
G_{4,i,j}(t,z)&=\det\left[\begin{matrix}
t_i&z_i\\
t_j&z_j \end{matrix}\right]  
 \,,& 1\le i<j\le n,\nonumber\\
G_{4,i,j,k}(x,t,z)&=\det\left[\begin{matrix}
t_i&z_i&x_i\\
t_j&z_j&x_j\\
t_k&z_k&x_k
\end{matrix}\right]\,,& 1\le i<j<k\le n,\nonumber
\end{align}
where the polynomials $G_{4,i,j,k}$ are defined if $n\geq 3$. Put
\begin{align}
\X_I&=\{w=(x,y,u,t,z)\in \X:G_{4,i,j}(t,z)= 0,\;1\le i<j\le n\},&\nonumber\\
\X_{II}&=\{w=(x,y,u,t,z)\in \X:G_0(x)=0,\;G_{4,i,j,k}(x,t,z)=0,\; 1\le i<j<k\le n\},&\nonumber\\
\W_I&=\{(w,\la)=(x,y,u,t,z,\la)\in \X\times \cc: z= \la t \},&\nonumber\\
\W_{II}&=\{(w,\la_1,\la_2)=(x,y,u,t,z,\la_1,\la_2)\in \X\times \cc\times \cc: G_0(x)=0,\;z=\la_1 t+\la_2 x\},&\nonumber\\
\Y_I&=\{w=(x,y,u,t,z)\in \Y:G_{4,i,j}(t,z)= 0,\;1\le i<j\le n\},&\nonumber\\
\Y_{II}&=\{w=(x,y,u,t,z)\in \Y:G_0(x)=0,\,G_{4,i,j,k}(x,t,z)=0,\; 1\le i<j<k\le n\},&\nonumber\\
\Z_{I}&=\{w=(x,y,u,t,z)\in \X:x\in \Gamma_I\},&\nonumber\\
\Z_{II}&=\{w=(x,y,u,t,z)\in \X:x\in\Gamma_{II}\},\nonumber&\\
\FF&=\{w=(x,y,u,t,z)\in \X:x\in \vf((0,1))\},&\nonumber
\end{align}
where the sets $\X_{II}$, $\W_{II}$ and $\Y_{II}$ are defined for $n\geq 3$.

Obviously $\X_I\subset \Y_I$ and $\X_{II}\subset\Y_{II}$. Moreover, any irreducible component of $\X_I$ is an irreducible component of $\Y_I$. The same holds for $\X_{II}$ and $\Y_{II}$.  Additionally, by  the Lagrange multipliers theorem and Facts \ref{Fact4}, \ref{fact1IandIIl} we immediately obtain

\begin{fact}\label{factZaropensubs}
{\rm (a)} Let 
$$
A_I=
\left\{w\in \X:\exists_{\la\in\cc}\;(w,\la)\in \W_I\right\}.
$$
If $\vf((0,1))\subset \Int \Omega$ then $\FF\subset \Z_I\subset A_I 
\subset \X_I\subset \Y_I$ and there exists an irreducible component $\X_{I,*}$ of $\overline{A_I}$ 
 which contains $\FF$ and is an  irreducible component of $\X_I$.

{\rm (b)} Let 
$$
A_{II}=
\left\{w\in \X:\exists_{\la_1,\la_2\in\cc}\;(w,\la_1,\la_2)\in \W_{II}\right\}.
$$
If $\vf((0,1))\subset \partial \Omega$ then $\FF\subset \Z_{II}\subset A_{II} 
\subset \X_{II}\subset \Y_{II}$ and there exists an irreducible component $\X_{II,*}$ of $\overline{A_{II}}$ 
 which contains $\FF$ and is an irreducible component of $\X_{II}$.
\end{fact}

\begin{proof}
From 
 Fact \ref{fact1IandIIl}(a) we have $\FF\subset \{(x,y,u,t,z)\in \X:x\in\Gamma_I\}\subset A_I$. 
Since all the polynomials $G_{4,i,j}$ vanish on $\X_I$, the vectors $t,\,z$ are linearly dependent provided $(x,y,u,t,z)\in\X_I$ for some $x,y,u$. So  $\X_I=\mathcal{X}_I\cup A_I$, 
 where
$$
\mathcal{X}_I=\{w=(x,y,u,t,z)\in \X_I: t=0\}.
$$
Obviously, the set $\mathcal{X}_I$ is contained in the hyperplane $H$ defined by $t=0$, and by Fact \ref{Fact4} we have $\FF\setminus H\ne \emptyset$, so $\overline{A_I}$ 
 has an irreducible component containing $\FF$ which is an irreducible component of $\X_I$. This gives  assertion (a).

Analogously, from Fact \ref{fact1IandIIl}(b)  we obtain $\FF\subset A_{II}$. 
 Moreover, the vectors $x,t,z$ are linearly dependent provided $(x,y,u,t,z)\in \X_{II}$ for some $y,u$, so $\X_{II}=\mathcal{X}_{II}\cup A_{II}$, 
  where
$$
\mathcal{X}_{II}=\{w=(x,y,u,t,z)\in \X_I: G_0(x)=0,\; G_{4,i,j}(x,t)=0,\;1\le i<j\le n\}.
$$
Obviously, $\mathcal{X}_{II}$ is contained in the set $W$ defined by $G_{4,i,j}(x,t)=0$, $1\le i<j\le n$. By Lemma \ref{ggg} we have $\FF\setminus W\ne \emptyset$, so as above, the set $A_{II}$ 
 has an irreducible component satisfying  (b).
\end{proof}


%
%

From Fact \ref{factZaropensubs} and Lemmas \ref{Fact3} and \ref{factY0smooth} and the definition of $\Y$ we have

\begin{fact}\label{factdegrees} 
$\delta(\X_{I,*})\le \delta(\Y_I)\le 2(2d-1)^{3n+1}$ and 
$\delta(\X_{II,*})\le \delta(\Y_{II})\le 2(2d-1)^{3n+1}$.
\end{fact}

The proofs of Theorems \ref{maintwr} and \ref{maintwrIII} consist  in showing that the projections of the sets $\X_{I,*}$ and $\X_{II,*}$ onto the space of $(y,u)\in \cc^2$ are proper algebraic subsets of $\cc^2$, since we have

\begin{lem}\label{lemfinishing}
If $Q\in\cc[y,u]$ is a nonzero polynomial of degree $D$ such that 
$$
Q(f(\vf(t)),|\nabla f(\vf(t))|^2)=0\quad\hbox{for }t\in[0,1),
$$
where  $\vf$ is the curve fulfilling \eqref{Lojgradinequality2}, then  

{\rm (a)} $\varrho_0(f)\leq 1-\frac{1}{D}$ if $D$ is  even,

{\rm (b)} $\varrho_0(f)\leq 1-\frac{1}{D+1}$ if $D$ is  odd. 
\end{lem}

\begin{proof}
Let $\ord_0 (f\circ \varphi)=M$ and $\ord_0 |\nabla f\circ\varphi|^2=K$. Then $M,K>0$ and
\begin{equation*}\label{eqQfnabla1}
\ord_0 (f\circ \varphi)^K=\ord_0|\nabla f\circ \varphi|^{2M},
\end{equation*}
i.e., $|f\circ \varphi|^{\frac{K}{2M}}\sim |\nabla f\circ\varphi|$ near zero\footnote{That is, there are $C_1,C_2>0$ such that $C_1|f\circ \varphi|^{\frac{K}{2M}}\le |\nabla f\circ\varphi|\le C_2|f\circ \varphi|^{\frac{K}{2M}}$ near zero.}, so by \eqref{Lojgradinequality2} we have 
\begin{equation}\label{eqro0}
\varrho_0(f)=\frac{K}{2M}.
\end{equation}
Then, by definitions of $M$ and $K$ 
 there exists a pair of different monomials $\alpha u^N y^S$ and $\beta u^{N_1}y^{S_1}$ of the polynomial $Q$ such that 
$$
N+S\le D \quad\hbox{and}\quad N_1+S_1\le D,
$$
and
$$
NK+SM=N_1K+S_1 M.
$$
Hence $N-N_1\ne 0$, $S_1-S\ne 0$, and 
$$
\frac{K}{2M}=\frac{S_1-S}{2(N-N_1)}.
$$
Since $M>0$, we have $\ord_0\nabla f\circ \varphi\le M-1$, and so $K\le 2M-2$, and  $\frac{K}{2M}<1$. 
On the other hand, $|S_1-S|, |N-N_1|\in \{1,\ldots, D\}$, so by \eqref{eqro0}, $\varrho_0(f)$ is estimated from  above by the maximal possible rational number less than $1$ with  numerator from the set $ \{1,\ldots, D\}$ and  denominator from $ \{2,4,\ldots, 2D\}$. Consequently, we obtain the assertion.
\end{proof}

\subsection{Proof of Theorem \ref{maintwr} in case I when $\vf((0,1))\subset \Int \Omega$} 

By the assumption \eqref{geberalassumption}, 
in the definition of $\Y$ one can take  the  polynomials 
\begin{multline}\label{eqformG3i}
K_{3,i}(x,y,u,z)
=\frac{\partial G}{\partial x_i}\left(x,y,u\right)\frac{\partial P}{\partial y}(x,y)-\frac{\partial G}{\partial y}\left(x,y,u\right) \frac{\partial P}{\partial x_i}(x,y) \\
\qquad-\left(\frac{\partial P}{\partial y}\left(x,y\right)\right)^3\cdot z_i
\end{multline}
instead of $G_{3,i}$, $1\le i\le n$; also 
   in the definitions of $\X_I$ and $\Y_I$  one can take  
\begin{equation*}\label{eqnewG4ij}
K_{4,i,j}(x,y,u)=\frac{\partial P}{\partial x_i}(x,y)\frac{\partial G}{\partial x_j}(x,y,u)\\-\frac{\partial P}{\partial x_j}(x,y)\frac{\partial G}{\partial x_i}(x,y,u)
\end{equation*}
 instead of $G_{4,i,j}$, $1\le i<j\le n$.

From the above and Fact \ref{factZaropensubs} 
we obtain the following fact. 

\begin{fact}\label{factKonv} For $x\in \Gamma_I$ and $v=(x,y,u)=(x,f(x),|\nabla f(x)|^2)$ we have
\begin{align}\label{eqK1}
P(v)&=0,&\\
\label{eqK2}
G(v)&=0,&\\
K_{4,i,j}(v)&=0,&\quad 1\le i<j\le n.\label{eqK4ij}
\end{align}
\end{fact}

Let $\V_{I,0}\subset \bf{M}$, where ${\bf M}= \cc^n\times\cc\times\cc$, be an algebraic set defined by 
the  system of equations \eqref{eqK1}--\eqref{eqK4ij}, and let
\begin{align}\nonumber
\Y^0_{I}&=\left\{(x,y,u,t,z)\in \Y_{I}:\frac{\partial P}{\partial y}(x,y)\ne 0\right\},\\
\nonumber
\V^0_{I}&=\left\{(x,y,u)\in \V_{I,0}:\frac{\partial P}{\partial y}(x,y)\ne 0\right\},\\
\nonumber
\V_I&=\overline{\V_{I}^0}.
\end{align}

We have the following fact (cf.  \cite[Fact 2.11]{KOSS}).

\begin{fact}\label{factWprojV}
The mapping 
\begin{equation*}\label{eqWprojV1}
\Y_I^0 \ni (x,y,u,t,z)\mapsto (x,y,u)\in\V^0_{I}
\end{equation*}
is a bijection.
\end{fact}

\begin{proof}
Taking any $(x,y,u,t,z)\in \Y^0_{I}$ (respectively $(x,y,u)\in\V^0_{I}$), by the  Implicit Function Theorem there are a neighbourhood $\Delta\subset \cc^n$ of $x$, a holomorphic function $g:\Delta\to\cc$ 
and neighbourhoods $U_1\subset \cc\times\cc\times\cc^n\times\cc^n$ and $U_2\subset \cc\times\cc$ of $(y,u,t,z)$ and $(y,u)$ respectively such that
\begin{align}
\nonumber
\Y_I^0\cap(\Delta\times U_1)&=\{(\zeta,g(\zeta),h(\zeta),\nabla g(\zeta),\nabla h(\zeta))\in\M:\zeta\in\Delta\cap V\},\\
\nonumber
\V_I^0\cap(\Delta\times U_2)&=\{(\zeta,g(\zeta),h(\zeta))\in {\bf M}: 
\zeta\in\Delta\cap V\},
\end{align}
where $h(\zeta)=\left(\frac{\partial g}{\partial x_1}(\zeta)\right)^2+\cdots+\left(\frac{\partial g}{\partial x_n}(\zeta)\right)^2$, and
$$
V=\{\zeta\in\Delta:K_{4,i,j}(\zeta,g(\zeta),h(\zeta))=0,\;1\le i<j\le n\}.
$$
 In particular, $g(x)=y$, $u=h(x)$, 
$t=\nabla g(x)$ and $z=\nabla h(x)$. Thus,  we obtain the assertion. 
\end{proof}

Let  $\E_I\subset  {\bf M}\times \cc$ 
 be the Zariski closure of the set 
\begin{equation*}
\E_{I,0}=\{(x,y,u,\la)\in  \Omega\times \rr\times\rr\times \rr:y=f(x),\; u=|\nabla f(x)|^2,\\
 \nabla|\nabla f(x)|^2=\la \nabla f(x)\}.
\end{equation*}

From Fact \ref{factZaropensubs}(a) we obtain

\begin{fact}\label{factEILagrmult}
There exists an irreducible component $\E_{I,*}$ of $\E_I$ which contains a Zariski open and dense subset $\mathcal{U}$ 
such that for any $(x,y,u,\la)\in \mathcal{U}$ there exist $t,z\in\cc^n$ such that  $(x,y,u,t,z)\in \X_{I,*}$ and in particular $z=\la t$. 
\end{fact}

\begin{proof}
The set $\E_I$ is the projection of the union of some irreducible components of $\W_I$ onto $(x,y,u,\la)\in{\bf M}\times\cc$. 
 So by  Fact \ref{factZaropensubs}(a) we obtain the assertion.
\end{proof}

Let
$$
\pi:{\bf M}\times \cc 
\ni (x,y,u,\la)\mapsto (x,y,u)\in {\bf M},
$$
let $\E_{I,*}$ be an irreducible component of $\E_I$ as in Fact \ref{factEILagrmult} and let
$$
\EE_I:=\overline{\pi(\E_{I,*})}.
$$ 

\begin{lem}\label{factcalV1proj}
The set $\EE_I$ is an irreducible component of the algebraic set $\V_{I}$. Moreover, $\EE_I$ contains a Zariski open and dense subset $\mathcal{U}_I$   such that $\mathcal{U}_I\subset\V_{I}^0\cap \pi(\E_{I,*})$, and any point $(x_0,y_0,u_0)\in\mathcal{U}_I$ has a neighbourhood $B\subset {\bf M}$ 
 such that $\V_I\cap B=\mathcal{U}_I\cap B$ and
\begin{equation}\label{eqformVgraph}
\mathcal{U}_I\cap B=\left\{\left(x,g(x),\left(\frac{\partial g}{\partial x_1}(x)\right)^2+\cdots+\left(\frac{\partial g}{\partial x_n}(x)\right)^2\right):x\in \Delta\cap V\right\}
\end{equation}
for some analytic set $V\subset \Delta$ with $x_0\in V$ and a holomorphic function $g:\Delta\to \cc$, where $\Delta\subset \cc^n$ is a neighbourhood of $x_0$.
 \end{lem} 

\begin{proof}
By Facts \ref{factZaropensubs}, \ref{factWprojV} and \ref{factEILagrmult} we have $\pi(\E_{I,0})\subset \V_{I}$, so  $\EE_I\subset \V_{I}$
  and $\EE_I$ is an algebraic subset of $\V_{I}$. Since any irreducible component of $\X_I$ is an irreducible component of $\Y_I$, the same holds for $\pi(\E_I)$ and $\V_I$,  because these sets are projections onto the space ${\bf M}$ 
 of some collections of irreducible components of $\X_I$ and $\Y_I$, respectively. In particular, this holds for $\EE_I$ and $\V_I$. This gives the first part of the assertion.  We prove the ``moreover'' part analogously to Fact \ref{factWprojV}. 
\end{proof}

Let
$$
\pi_y:\EE_I\ni v=(x,y,u)\mapsto y\in\cc,
$$
$$
\pi_u:\EE_I\ni v= (x,y,u)\mapsto u\in \cc.
$$
We have the following lemma (cf. \cite[Lemmas 2.12, 2.14]{KOSS}).

\begin{lem}\label{lemfinitenesscriticalvalues}
For  generic $y_0\in \cc$, i.e., for any $y_0\in\cc$ off a finite set,  the function $\pi_u$ 
is constant on each connected component of $(\pi_y)^{-1}(y_0)$. 
\end{lem}

\begin{proof}
If $\dim \EE_I=0$ or $\dim (\pi_y)^{-1}(y)\leq 0$ for  generic $y\in\cc$, then the assertion holds. Assume that $\dim \EE_I>0$ and $\dim(\pi_y)^{-1}(y)>0$ for generic $y\in\cc$. Then by Lemma \ref{factcalV1proj}, and under the notations of this lemma, we have 
  $\overline{\pi_y(\mathcal{U}_I)}=\overline{\pi_y(\EE_I)}=\cc$ and  $(\pi_y)^{-1}(y)\cap \mathcal{U}_I\ne \emptyset$ for  generic $y\in \cc$. 

Take any $y_0\in \cc$ such that $(\pi_y)^{-1}(y_0)\cap \mathcal{U}_I\ne \emptyset$. Take any $x_0\in\cc^n$ and $u_0\in\cc$ such that $(x_0,y_0,u_0)\in \mathcal{U}_I$. By Lemma \ref{factcalV1proj} there exist a neighbourhood $B\subset {\bf M}$ 
of $(x_0,y_0,u_0)$ and a holomorphic function $g:\Delta\to\cc$, where $\Delta\subset \cc^n$ is a neighbourhood of $x_0$, such that \eqref{eqformVgraph} holds for some analytic set $V\subset \Delta$.

Take any smooth curve $\gamma:[0,1]\to \Delta\cap V$ such that $g(\gamma(t))=y_0$ for $t\in [0,1]$. Let $h(x)=\left(\frac{\partial g}{\partial x_1}(x)\right)^2+\cdots+\left(\frac{\partial g}{\partial x_n}(x)\right)^2 $      for $x\in \Delta$  and take a function $u:[0,1]\to\cc$ defined by
$$
u(t)=h\circ \gamma(t).
$$
Observe that the function $u$ is constant. Indeed, by definition of $\mathcal{U}_I$ we see that for any $x\in  \Delta\cap V$ there exists $\la_x\in\cc$ such that 
$$
\nabla h(x)=\la_x \nabla g(x).
$$
So,
\begin{equation*}\label{equprime}
u'(t)=\la_{\gamma(t)}\langle \nabla g(\gamma(t)),\overline{\gamma'(t)}\rangle\quad\hbox{for }t\in[0,1],
\end{equation*}
where $\langle \cdot,\cdot\rangle$ denotes the standard scalar product in $\cc^n$. 
Since $g(\gamma(t))=y_0$ for $t\in[0,1]$, we have $\langle \nabla g(\gamma(t)),\overline{\gamma'(t)}\rangle=0$, and consequently $u'(t)=0$ for $t\in[0,1]$ and $u$ is constant. Summing up, the function $\pi_u$ is constant on each connected component of $(\pi_y)^{-1}(y_0)\cap \mathcal{U}_I$. 

Since $\mathcal{U}_I$ is a Zariski open and dense subset of $\EE_I$, any irreducible component of $\EE_I\setminus \mathcal{U}_I$ has dimension smaller than the dimension of 
 $\EE_I$, and for generic $y\in\cc$ any irreducible component $A$ of the fibre $\pi_y^{-1}(y)$ has a dense subset of the form $A\cap \mathcal{U}_I$ (see \cite[Chapter 3]{Mumford}). Then by the above we obtain the assertion.
\end{proof}

Since $\Gamma$ is an infinite set, it follows that $\dim\E_{I,0}\ge 1$, so by Fact \ref{factWprojV}, $\dim {\E_I}\ge 1$, and since $d=\deg P\ge 2$,  Lemma \ref{Fact3} and  the definition of $\V_I$ yield
 $\delta (\EE_I)\le d(3d-2)^n$, where  $\delta(\EE_I)$ is the total degree of  $\EE_I$. 
So, from Lemma \ref{lemfinitenesscriticalvalues}, the closure of the projection of $\EE_I$,  $W=\overline{\{(y,u)\in\cc^2:\exists_{x\in\cc^n}\;(x,y,u)\in \EE_{I}\}}$, is a proper algebraic subset of $\cc^2$ and by Fact \ref{Fact2}, $\delta(W)\leq \delta(\EE_I)$. Then 
there exists a nonzero polynomial $Q\in\cc[y,u]$ such that
\begin{equation*}\label{eqestdegQ}
\deg Q\le d(3d-2)^n\leq \mathcal{R}(n,d)-1
\end{equation*}
and $Q(y,u)=0$ for $(x,y,u)\in \EE_I$. In particular, $Q(f(\vf(t)),|\nabla f(\vf(t))|^2)=0$ for $t\in [0,1)$. 
Since $D=d(3d-2)^n$ may be  odd,  by Lemma \ref{lemfinishing}(b) we obtain the assertion of Theorem \ref{maintwr} in  case I.

\subsection{Proof of Theorem \ref{maintwr} in  case II when  $\varphi\big([ 0,1)\big)\subset\partial\Omega$}
For any $x\in\partial \Omega\setminus f^{-1}(0)$ sufficiently close to $f^{-1}(0)$ the tangent spaces to $\partial \Omega$ and $f^{-1}(f(x))$ are transversal, as  shown in Lemma \ref{ggg}.

We will 
 prove Theorem \ref{maintwr} in two dimensions and in the multidimensional case  separately.

\noindent{\it Proof of Theorem \ref{maintwr} in  case II for $n=2$.}  Take a polynomial  $G\in \cc[x,y,u]$, where $x=(x_1,x_2)$ and  $y$, $u$ are single variables,  defined by \eqref{eqdefG}, i.e., $
G(x,y,u)=\sum_{i=1}^2\left(\frac{\partial P}{\partial x_i}(x,y)\right)^2-\left(\frac{\partial P}{\partial y}(x,y)\right)^2 \cdot u$.
Let 
\begin{align}\nonumber
\V_{II,0}&=\{(x,y,u)\in \cc^2\times \cc\times \cc:P(x,y)=0,\;G_0(x)=0,\;G(x,y,u)=0\},\\
\nonumber 
\V_{II}^0&=\left\{(x,y,u)\in \V_{II,0}:\frac{\partial P}{\partial y}(x,y)\ne 0\right\},\\
\nonumber
\V_{II}&=\overline{\V^0_{II}}.
\end{align}
Then for any $x\in \Gamma\cap\partial \Omega$ we have $(x,f(x),|\nabla f(x)|^2)\in \V_{II}$. Consequently, 
$$
(\varphi(t),f(\varphi(t)),|\nabla f(\varphi(t))|^2)\in \V_{II}\quad\hbox{for }t\in [0,1).
$$
In particular, $\dim \V_{II}\ge 1$ and by Fact \ref{Fact1} we have $\delta (\V_{II})\le 2d(2d-1)$.

Since $P$ is an irreducible polynomial of positive degree with respect to $y$, for any $y\in\cc\setminus\{0\}$ sufficiently close to $0$ the set $\{x\in\cc^2:P(x,y)=0,\,G_0(x)=0\}$ is finite, so the set 
$
\{(x,u)\in \cc^2\times \cc:(x,y,u)\in \V_{II}\}
$ 
is also finite. Then the projection 
$$
W=\{(y,u)\in\cc^2:\exists_{x\in\cc^2}(x,y,u)\in \V_{II}\}
$$
is contained in a proper algebraic subset of $\cc^2$. By Fact \ref{Fact2},
$$
\delta(\overline{W})\le 2d(2d-1)\le \mathcal{R}(n,d).
$$
Then there exists a nonzero polynomial $Q\in\cc[y,u]$ of degree $\deg Q\le \delta(\overline{W})\le \mathcal{R}(n,d)$ which vanishes on $W$. Since $2d(2d-1)$ is  even, by  Lemma \ref{lemfinishing}(a) we obtain the assertion of Theorem \ref{maintwr} in  case II for $n=2$.
\hfill$\square$

Let us consider the case $n\ge 3$. 
Let $\varepsilon>0$ be as in Lemma \ref{ggg}. 

By the assumption \eqref{geberalassumption}, 
 in the definition of the set $\Y$  one can take the polynomials  $K_{3,i}$ of the form \eqref{eqformG3i} instead of $G_{3,i}$; also, in the definitions of $\X_{II}$ and $\Y_{II}$,   
one can take the polynomials
\begin{equation*}
K_{4,i,j,k}(x,y,u)=\left|\begin{matrix}
\frac{\partial P}{\partial x_i}(x,y)&\frac{\partial G}{\partial x_i}\left(x,y,u\right)& x_i\\
\frac{\partial P}{\partial x_j}(x,y)&\frac{\partial G}{\partial x_j}\left(x,y,u\right)& x_j\\
\frac{\partial P}{\partial x_k}(x,y)&\frac{\partial G}{\partial x_k}\left(x,y,u\right) &x_k
\end{matrix}\right|,
\end{equation*}
instead of $G_{4,i,j,k}$ for $1\le i<j<k\le n $, where $G$ is defined in \eqref{eqdefG}. 
Then 
\begin{align}
\X_{II}&=\{w=(x,y,u,t,z)\in \X:G_0(x)=0,\;K_{4,i,j,k}(x,y,u)=0,\; 1\le i<j<k\le n\},\nonumber\\
\Y_{II}&=\{w=(x,y,u,t,z)\in \Y:G_0(x)=0,\,K_{4,i,j,k}(x,y,u)=0,\; 1\le i<j<k\le n\}.\nonumber
\end{align}

Let $\V_{II,0}\subset {\bf M}$, where ${\bf M}=\cc^n\times\cc\times\cc$, 
 be  the algebraic set defined by 
\begin{multline*}
\V_{II,0}=\{(x,y,u)\in {\bf M} : 
P(x,y)=0,\;G_0(x)=0,\;G(x,y,u)=0,\\
K_{4,i,j,k}(x,y,u)=0,\;1\le i<j<k\le  n \}
\end{multline*}
and let 
\begin{align}
\nonumber 
\Y^0_{II}&=\left\{(x,y,u,t,z)\in \Y_{II}:\frac{\partial P}{\partial y}(x,y)\ne 0\right\},\\
\nonumber
\V^0_{II}&=\left\{(x,y,u)\in \mathbb{V}_{II,0}:\frac{\partial P}{\partial y}(x,y)\ne 0\right\},\\
\nonumber
\V_{II}&=\overline{\V_{II}^0}.
\end{align}

By an analogous argument to the proof of Fact \ref{factWprojV} we obtain

\begin{fact}\label{factWprojVII}
The mapping 
\begin{equation*}\label{eqWprojV1}
\Y_{II}^0 \ni (x,y,u,t,z)\mapsto (x,y,u)\in\V^0_{II}
\end{equation*}
is a bijection.
\end{fact}




Let  $\E_{II}\subset {\bf M}\times\cc^2$ 
 be the Zariski closure of the set 
\begin{multline*}
\E_{II,0}=\{(x,y,u,(\la_1,\la_2))\in \partial \Omega\times \rr\times\rr\times \rr^2:y=f(x),\; u=|\nabla f(x)|^2,\\
 \nabla|\nabla f(x)|^2=\la_1 \nabla f(x)+\la_2 x\}.
\end{multline*}
By a similar argument to the proof of Fact \ref{factEILagrmult}, from Fact \ref{factZaropensubs}(b) we obtain

\begin{fact}\label{factEILagrmultII}
There exists an irreducible component $\E_{II,*}$ of $\E_{II}$ which contains a Zariski open, dense subset $\mathcal{U}$ 
such that for any $(x,y,u,\la_1,\la_2)\in \mathcal{U}$ there exist $t,z\in\cc^n$ such that  $(x,y,u,t,z)\in \X_{II,*}$ and in particular $z=\la_1 t+\la_2 x$. 
\end{fact}


Let
$$
\pi':{\bf M}
\times \cc^2\ni (x,y,u,(\la_1,\la_2))\mapsto (x,y,u)\in {\bf M},
$$
and let
$$
\EE_{II}=\overline{\pi'(\E_{II,*})}.
$$

By an analogous argument to the proof of Lemma \ref{factcalV1proj} we obtain 

\begin{lem}\label{factcalV1projII}
The set $\EE_{II}$ 
 is an 
  irreducible component of the algebraic set $\V_{II}$. Moreover, $\EE_{II}$ 
  contains a Zariski open and dense subset $\mathcal{U}_{II}$  
   such that $\mathcal{U}_{II}\subset\V^0_{II}\cap \pi'(\E_{II,*})$ and any point $(x_0,y_0,u_0)\in\mathcal{U}_{II}$ has a neighbourhood $B\subset {\bf M}$ 
    such that $\V_{II}\cap B=\mathcal{U}_{II}\cap B$ and
\begin{equation}\label{eqformVgraphII}
\mathcal{U}_{II}\cap B=\left\{\left(x,g(x),\left(\frac{\partial g}{\partial x_1}(x)\right)^2+\cdots+\left(\frac{\partial g}{\partial x_n}(x)\right)^2\right):x\in \Delta\cap V\right\}
\end{equation}
for some analytic set $V\subset \Delta$, where $x_0\in V$ and $G_0$ vanishes on $V$, and a holomorphic function $g:\Delta\to \cc$, where $\Delta\subset \cc^n$ is a neighbourhood of $x_0$.
 \end{lem} 

Let
$$
\pi_y:\EE_{II}\ni v=(x,y,u)\mapsto y\in\cc,
\quad\
\pi_u:\EE_{II}\ni v= (x,y,u)\mapsto u\in \cc.
$$

We have the following lemma (cf. Lemma \ref{lemfinitenesscriticalvalues} and \cite[Lemmas 2.12, 2.14]{KOSS}).

\begin{lem}\label{lemfinitenesscriticalvaluesII}
For generic $y_0\in \cc$ the function $\pi_u$ 
is constant on each connected component of $(\pi_y)^{-1}(y_0)$. 
\end{lem}

\begin{proof} As in the proof of Lemma \ref{lemfinitenesscriticalvalues}, we may assume that $\dim \EE_{II}>0$ and $\dim\,(\pi_y)^{-1}(y)>0$ for generic $y\in\cc$. Then by Lemma \ref{factcalV1projII}, and under the notations of that lemma, 
  $\overline{\pi_y(\mathcal{U}_{II})}=\overline{\pi_y(\EE_{II})}=\cc$ and  $(\pi_y)^{-1}(y)\cap \mathcal{U}_{II}\ne \emptyset$ for generic $y\in \cc$. 

Take any $y_0\in \cc$ such that $(\pi_y)^{-1}(y_0)\cap \mathcal{U}_{II}\ne \emptyset$. Take any $x_0\in\cc^n$ and $u_0\in\cc$ such that $(x_0,y_0,u_0)\in \mathcal{U}_{II}$. By Lemma \ref{factcalV1projII} there exist a neighbourhood $B\subset \cc^n\times\cc\times\cc$ of $(x_0,y_0,u_0)$ and a holomorphic function $g:\Delta\to\cc$, where $\Delta\subset \cc^n$ is a neighbourhood of $x_0$, such that \eqref{eqformVgraphII} holds for some analytic set $V\subset \Delta$ such that $G_0$ vanishes on $V$.

Take a smooth curve $\gamma=(\gamma_1,\ldots,\gamma_n):[0,1]\to \Delta\cap V$ such that $g(\gamma(t))=y_0$. Then  
\begin{equation}\label{eqgammagammaprime}
G_0(\gamma(t))=0 \quad\hbox{for }t\in [0,1] .
\end{equation}
Let $h(x)=\left(\frac{\partial g}{\partial x_1}(x)\right)^2+\cdots+\left(\frac{\partial g}{\partial x_n}(x)\right)^2$, $x\in \Delta$. Take a function $u:[0,1]\to\cc$ defined by
$$
u(t)=h\circ \gamma(t),\quad t\in [0,1].
$$
Observe that the function $u$ is constant. Indeed, by definition of  $\mathcal{U}_{II}$, for any $x\in  \Delta\cap V$ there exist  $\la_{1,x},\la_{2,x}\in\cc$ such that 
$$
\nabla h(x)=\la_{1,x} \nabla g(x)+\la_{2,x}x.
$$
So
\begin{equation*}\label{equprime}
u'(t)=\la_{1,\gamma(t)}\langle \nabla g(\gamma(t)),\overline{\gamma'(t)}\rangle+\la_{2,\gamma(t)}\langle \gamma(t),\overline{\gamma'(t)}\rangle \quad\hbox{for }t\in[0,1].
\end{equation*}
Since $g(\gamma(t))=y_0$, we have $ \langle \nabla g(\gamma(t)),\overline{\gamma'(t)}\rangle=0$ for $t\in[0,1]$. Moreover, by \eqref{eqgammagammaprime} we have $\langle \gamma(t),\overline{\gamma'(t)}\rangle=0$ for $t\in[0,1]$. 
 Consequently, $u'(t)=0$ for $t\in[0,1]$ and $u$ is constant. Summing up, the function $\pi_u$ is constant on each connected component of $(\pi_y)^{-1}(y_0)\cap \mathcal{U}_{II}$. Since $\mathcal{U}_{II}$ is a dense subset of $\EE_{II}$, we obtain the assertion.
\end{proof}


Since $\Gamma$ is an infinite set, we have $\dim\E_{II,0}\ge 1$, so by Fact \ref{factWprojVII}, $\dim {\E_{II}}\ge 1$, and since $d=\deg P\ge 2$,  Lemma \ref{Fact3} and  the definition of $\V_{II}$ yield $\delta (\EE_{II})\le d(3d-2)^n$. So, from Lemma \ref{lemfinitenesscriticalvaluesII}, the closure of the projection  of $\EE_{II}$, 
$$W=\overline{\{(y,u)\in\cc^2:\exists_{x\in\cc^n}\;(x,y,u)\in \EE_{II}\}},$$ 
is a proper algebraic subset of $\cc^2$ and $\delta({W})\leq \delta(\EE_{II})$. Then there exists a nonzero polynomial $Q\in\cc[y,u]$ such that 
$\deg Q\le 2(3d-2)^n \leq\mathcal{R}(n,d)-1$ 
and $Q(y,u)=0$ for $(x,y,u)\in \EE_{II}$. Since $D=2(3d-2)^n$ is an even number, by Lemma \ref{lemfinishing}(a) we obtain the assertion of Theorem \ref{maintwr} in  case II.

\subsection{Proof of Theorem \ref{maintwrIII}}\label{roz2III}

Analogously to the proof of Lemma \ref{lemfinitenesscriticalvalues}, 
we prove that the set
$$
W=\overline{\{(y,u)\in\cc^2:\exists_{x\in\cc^n}\;\exists_{t\in\cc^n}\;\exists_{z\in\cc^n}\;(x,y,u,t,z)\in \Y_{I}\}}
$$
is a proper algebraic subset of $\cc^2$. Moreover, by Fact \ref{factdegrees} 
 we have $\delta(W)\le \delta(\Y_{I})\le 2d(2d-1)$ if $n=1$ and $\delta(W)\le \delta(\Y_I)\le 2(2d-1)^{3n+1}$ for $n\ge 2$. Then by  Lemma \ref{lemfinishing}(a) we obtain the assertion of Theorem \ref{maintwrIII} in case I. An analogous argument gives the assertion in  case II.

\end{document}